\documentclass[a4paper,10pt]{article}

\usepackage{mathtools}

\usepackage{algorithmic}

\usepackage{algorithm}

\usepackage{amsthm}

\usepackage{amssymb}

\usepackage{amsmath}

\usepackage{enumitem}

\usepackage{float}

\theoremstyle{plain}

\newtheorem{thm}{Theorem}[section]

\newtheorem{prop}[thm]{Proposition}

\newtheorem{lem}[thm]{Lemma}

\newtheorem{cor}[thm]{Corollary}

\theoremstyle{definition}

\newtheorem{definition}[thm]{Definition}

\theoremstyle{remark}

\newtheorem{remark}[thm]{Remark}

\numberwithin{paragraph}{thm}

\numberwithin{equation}{section}

\floatstyle{boxed}

\restylefloat{figure}

\begin{document}

\title{Local And Global Colorability of Graphs}

\author{Noga Alon
\thanks{
Sackler School of Mathematics and
Blavatnik School of Computer Science,
Tel Aviv University,
Tel Aviv 69978, Israel and School of Mathematics,
Institute for Advanced Study, Princeton,
NJ 08540. Email: nogaa@tau.ac.il.
Research supported in part by a
USA-Israeli
BSF grant, by an ISF grant,
by the Israeli I-Core program
and by the Oswald Veblen Fund.}
\and 
Omri Ben-Eliezer
\thanks{Blavatnik School of
Computer Science, Tel Aviv University, Tel Aviv 69978, Israel.
Email: omrib@mail.tau.ac.il.}}

\maketitle

\begin{abstract}
It is shown that for any fixed $c \geq 3$ and $r$, the maximum
possible
chromatic number of a graph on $n$ vertices in which every subgraph
of radius at most $r$ is $c$ colorable is $\tilde{\Theta}\left(n ^
{\frac{1}{r+1}} \right)$ (that is, $n^\frac{1}{r+1}$ up to a factor
poly-logarithmic in $n$). The proof is based on a careful analysis of
the
local and global colorability of random graphs and implies, in
particular, that a random $n$-vertex graph with the right edge
probability has typically a chromatic number as above and yet 
most balls of radius $r$ in it are $2$-degenerate.
\end{abstract} 

%%%%%%%%%%%%%%%%%%%%%%%%%%%%%%%%%%%%%%%%%%%%%%%%%%%%%%%%%%%%%%%%%%%%%%

%%%%%%%%%%%%%%%%%%%%%%%%%%%%%%%%%%%%%%%%%%%%%%%%%%%%%%%%%%%%%%%%%%%%%%

\section{Introduction}

%%%%%%%%%%%%%%%%%%%%%%%%%%%%%%%%%%%%%%%%%%%%%%%%%%%%%%%%%%%%%%%%%%%%%%

\subsection{Notation and Definitions}

For a simple undirected graph $G=(V,E)$ denote by $d(u,v)$ the
\emph{distance} between the vertices $u,v\in V$. The \emph{degree} of
a
vertex $v \in V$, denoted by $\deg(v)$, is the number of its
neighbours 
in $G$. A subset $V' \subseteq V$ is \emph{independent}
if no edge of $G$ has both of its endpoints in $V'$.  The \emph{chromatic
number}
of $G$, denoted by $\chi(G)$, is the minimal number of independent
subsets
of $V$ whose union covers $V$.  A graph is \emph{$k$-degenerate} if the
minimum degree of every subgraph of it is at most $k$. In particular,
a $k$-degenerate graph is $k+1$-colorable. We will work with random
graphs $G_{n,p}$ in the Erd\H{o}s-R\'enyi model, where there are $n$
labelled vertices and each edge is included in the graph with
probability
$p$,  independently of all other edges. We say that a property of $G$
holds \emph{with high probability (w.h.p.\@)} if this  property holds
with probability that tends to 1 as $n$ tends to $\infty$. In this
paper
we are only interested in graphs with large chromatic number $\ell$.
It
will be therefore equivalent to say that a property holds w.h.p.\@ if
its
probability tends to 1 as $\ell$ tends to $\infty$. 

%A graph property
%$Q$ is said to be \emph{monotone increasing} if any graph that has $Q$
%will still have $Q$ if we add edges to it; A \emph{monotone decreasing}
%property is defined similarly.

Consider the following definition of $r$-local colorability:
\begin{definition}
\label{def:r_ball}
Let $r$ be a positive integer. Let ${U_r}(v,G)$ be the ball with 
radius $r$ around $v\in V$ in $G$ (i.e.\@ the induced subgraph 
on all vertices in $V$ whose distance from $v$ is $\leq r$). 
Let 
\begin{equation}
\ell{\chi_r}(G)=\max_{v\in V}\chi({U_r}(v,G))
\end{equation}
denote the \emph{$r$-local chromatic number} of $G$.
\end{definition}
We also say that $U_r(v,G)$ is the $r$-ball around $v$ in $G$.
Finally, we define the main quantity discussed in this paper.
\begin{definition}
\label{def:f_c(n,r)}
For $\ell \geq c \geq 2$ and $r > 0$ let
$f_c(\ell,r)$ be the greatest integer $n$ such that every graph 
on $n$ vertices whose $r$-local chromatic number 
is $\leq c$ is $\ell$-colorable.
\end{definition} 

In other words, $f_c(\ell,r) + 1$ is the minimal number of vertices 
in a non-$\ell$-colorable graph in which every $r$-ball is
$c$-colorable. 
Note that $f_{c_1}(\ell,r) \leq f_{c_2}(\ell,r)$ for $c_1 \geq c_2$. 

Definitions \ref{def:r_ball} and \ref{def:f_c(n,r)} appear
explicitly in the paper of Bogdanov \cite{bogdanov14}, but the 
quantity $f_c(\ell,r)$ itself has been investigated well before 
(see sections \ref{sec:bg}, \ref{sec:conc_remarks} for more details).

The main goal of this paper is to estimate $f_c(\ell, r)$ for fixed
$c,r$ as $\ell$ tends to $\infty$. The main result is
an upper bound tight up to a
polylogarithmic factor for $f_c(\ell, r)$ for all fixed $c \geq 3$
and $r$.

%%%%%%%%%%%%%%%%%%%%%%%%%%%%%%%%%%%%%%%%%%%%%%%%%%%%%%%%%%%%

\subsection{Background and our contribution}

\label{sec:bg}

%%%%%%%%%%%%%%%%%%%%%%%%%%%%%%%%%%%%%%%%%%%%%%%%%%%%%%%%%%%%

Fix an $r > 0$. Somewhat surprisingly, the gap between $f_2(\ell, r)$
and
$f_3(\ell, r)$ might be much bigger than the gap between $f_3(\ell,
r)$ and $f_c(\ell, r)$ for any other fixed $c \geq 3$. Here is a
short
background on previous results regarding $f_c(\ell, r)$ for fixed $c$
and $r$ and large $\ell$ and our contributions to these problems.

\subsubsection*{Known upper bounds for $f_c(\ell, r)$ with fixed
$c,r$, large $\ell$}
Erd\H{o}s \cite{erdos59} showed that for sufficiently large $m$ there
exists a graph $G$ with $m^{1+1/2k}$ vertices, that neither contains
a
cycle of length $\leq k$ nor an independent set of size $m$. As an
easy
consequence, $G$ is not $m^{1/2k}$-colorable. Put $k=2r+1,
\ell=m^{1/2k}$
and note that $G$ has $n=m^{1+1/2k} = \ell^{2k+1} = \ell^{4r+3}$
vertices
and $\ell\chi_r(G) \leq 2$ but is not $\ell$-colorable. Hence
$$
f_2(\ell, r) < \ell^{4r+3}.
$$
A better estimate follows from the results of Krivelevich in
\cite{Kriv95}. Indeed, Theorem 1 in his paper implies that
there exists an absolute positive constant $c$ so that
\begin{equation}
\label{e91}
f_2(\ell, r) < (c \ell \log {\ell})^{2r}
\end{equation}
An upper bound for $f_3(\ell, r)$ can be derived from another
result by Erd\H{o}s \cite{erdos62}. Erd\H{o}s worked with random
graphs in the $G_{n,m}$ model, in which we consider random graphs
with
$n$ vertices and exactly $m$ edges. He showed that with probability
$> 0.8$ and for $k \leq O(n^{1/3})$ large enough, $G_{n,kn}$ is not
$\frac{k}{\log{k}}$-colorable but every subgraph spanned 
by ${O}(nk^{-3})$
vertices is $3$-colorable.

It is easy to show that with high probability every $r$-ball in
$G_{n,kn}$ has $O(k)^{r}$ vertices (later we prove and apply a
similar
result for graphs in the $G_{n,p}$ model). Combining the above
results and taking $k = 2\ell \log{\ell}$, $n = O(k)^{r+3} = O(\ell
\log{\ell})^{r+3}$, it follows that 
with positive probability the graph $G_{n,kn}$ is not
$\ell$-colorable but every $r$-ball (and in fact every subgraph on
$O(nk^{-3}) = O(k)^{r}$ vertices) is $3$-colorable. Hence there
exists
$\beta > 0$ such that:
\begin{equation}
f_c(\ell, r) \leq f_3(\ell, r) \leq (\beta \ell \log{\ell})^{r+3}
\end{equation}
for large $\ell$, fixed $r \geq 3$ and for $c \geq 3$.

\subsubsection*{Known lower bounds 
for $f_c(\ell, r)$ with fixed $c,r$, large $\ell$}
Bogdanov \cite{bogdanov14} 
showed that for all $r>0$ and $\ell \geq c \geq 2$:
\begin{equation}
\label{eq:bogd}
f_c(\ell,r) \geq \frac{(\ell/c+r/2)(\ell/c+r/2+1)
\ldots(\ell/c+3r/2)}{(r+1)^{r+1}} \geq 
\left( \frac{\ell / c + r / 2} {r+1} \right)^{r+1}
\end{equation}
When $c$ and $r$ are fixed, 
\eqref{eq:bogd} implies that $f_c(\ell,r) = \Omega(\ell^{r+1})$.

\subsubsection*{A special case - $f_c(\ell, 1)$ for fixed $c$, large
$\ell$}
It is not difficult to prove that
$f_2(\ell,1)=\Theta(\ell^2\log{\ell})$,
using
the known fact that the Ramsey number $R(t,3)$ is $\Theta(t^2 /
\log{t})$
(see \cite{ajtai80}, \cite{kim95}). In Section \ref{sec:r_is_1} we
extend
this result to every fixed $c \geq 2$, showing that $f_c(\ell, 1) =
\Theta(\ell^2 \log{\ell})$ for any fixed $c \geq 2$.

\subsubsection*{The main contribution}  
The main result in this paper is
an improved upper bound for $f_3(\ell, r)$. We show that for fixed
$r > 0$: 
\begin{equation} 
\label{eq:main_contrib} 
f_3(\ell, r) \leq \left( 10 \ell \log{\ell} \right) ^ {r+1} 
\end{equation} 
Fix $r$ and $c \geq 3$.  By the result above (together with
\ref{eq:bogd})  it follows that
there exists a constant $\delta = \delta(r,c)$ such that
\begin{equation} 
(\delta \ell)^{r+1} \leq f_c(\ell, r) \leq f_3(\ell,
r) \leq (10 \ell \log{\ell})^{r+1} 
\end{equation} 
The last result
determines, up to a logarithmic factor, the maximum possible 
chromatic number
$M_{c,r}(n)$ of a graph on $n$ vertices in which every $r$-ball is
$c$-colorable:
\begin{equation} 
a \frac{n^{\frac{1}{r+1}}}{\log{n}}
\leq M_{c,r}(n) \leq b_{c,r} n^\frac{1}{r+1} 
\end{equation} 
for suitable positive constants $a$, $b_{c,r}$. \\
Note that for $c=2$ the best known estimates are weaker, 
namely it is only known that
$$
\Omega(\frac{n^{1/(2r)}}{\log n}) \leq M_{2,r} 
\leq O(n^{1/(r+1)}).
$$

% Maybe add a remark about the explicit construction

\subsection{Paper Structure}
The rest of  the paper is organized as follows:
\begin{itemize} 
\item In Section \ref{sec:grad_reveal} we present the 
basic approach of gradually
revealing information on a random graph. Two examples of this 
are given. Both will be useful in subsequent sections.

\item In Section \ref{sec:f_5} we give an upper bound for 
$f_5(\ell,r)$
for fixed $r$ and large $\ell$ using the random graph $G_{n,p}$
with $n = (10 \ell \log{\ell})^{r+1}$ and $p = \frac{3}{10} (10 \ell
\log{\ell})^{-r}$. It is shown that with high probability, all
$r$-balls
in the graph are $4$-degenerate.

\item In Section \ref{sec:f_4}, the same upper bound is obtained for
$f_4(\ell, r)$. It is shown that most $r$-balls in the above graph
are
$4$-colorable. Deleting the center of every non-$4$-colorable
$r$-ball
results in a graph with $r$-local chromatic number $ \leq 4$ and
chromatic
number $> \ell$ with positive probability.

\item Section \ref{sec:f_3} includes the proof of the main result 
of the paper. It is shown that typically most $r$-balls in the 
above graph are $2$-degenerate. 
This proof is much harder than the previous one. Again we delete
the center of every non-$2$-degenerate $r$-ball to obtain a graph
with
$r$-local chromatic number at most $3$ and chromatic number $> \ell$
with positive probability.

Note that the result in this section is stronger than those in the
previous two sections.  Still, we prefer to include all three as
each of the results has its merits: indeed, to get
local $5$ colorability it suffices to consider random graphs 
with no changes. Getting local $4$-colorability requires some
modifications in the random graph, but the proof is very short.

Getting local $3$-colorability is significantly more complicated,
and is proved by a delicate exposure of the information 
about the edges of the random
graph considered.

\item In Section \ref{sec:c_non_const} we extend the result from
Section
\ref{sec:f_3} to large values of $c$.

\item In Section \ref{sec:r_is_1} it is shown
that $f_c(\ell, 1) = \Theta(\ell^2 \log{\ell})$ for any fixed $c \geq
2$. 

\item 
The final Section \ref{sec:conc_remarks} contains some concluding
remarks including a discussion of what can be proved about
the behaviour of $f_c(\ell,r)$ for
non-constant
values of $r$.

\end{itemize}

%%%%%%%%%%%%%%%%%%%%%%%%%%%%%%%%%%5555

\section{Gradually Revealing the Random Graph}

\label{sec:grad_reveal}

%%%%%%%%%%%%%%%%%%%%%%%%%%%%%%%%%%5555
In random graphs of the $G_{n,p}$ model the edges can be examined
(that is, accepted to the graph or rejected from it) in any order.
This
fact can be used to reveal some of the information regarding the
graph,
while preserving the randomness of other information. Two examples of
this basic approach are shown below, 
both  will be used later in this paper.

\subsection{Spanning tree with root}
\label{subsec:spantree}
Let $r > 0$. This model first determines the vertices of $U_r(v,G)$
while also revealing a spanning tree for this subgraph, and only then
continues to reveal all other edges of the graph.

Choose a root vertex $v$. Let $L_i = L_i(v,G)$ denote \emph{the
$i$-th
level} with respect to $v$ in $G$ - that is, the set of all vertices
of distance $i$ from $v$. Trivially, $L_0(v,G) = \{ v \}$. Also
define
$L_{\leq i} = L_{\leq i}(v,G) = \bigcup_{j=0}^{i} L_j(v,G)$.

Assuming $L_i$ is already known and $T$ is constructed up to the
$i$-th
level, reveal $L_{i+1}$ and expand $T$ as follows: for every $u \in
V$
not in the tree, examine the possible edges from $u$ to $L_i$ one by
one. Stop either when an examined edge from $u$ to $L_i$ is accepted
to
the graph (in this case, $u \in L_{i+1}$ and the accepted edge is
added
to the tree) or when all possible edges from $u$ to $L_i$ are
rejected
(here $u \notin L_{i+1}$). An easy induction shows that the newly
added
vertices are exactly all vertices of $L_{i+1}$.

Stop this process after $L_r$ is revealed. The remaining unexamined
edges can later be examined in any order. Let $T = T(v)$ be the
spanning
tree of $U_r(v,G)$ and let $R = R(v) = U_r(v,G) \setminus T(v)$ (i.e.
R
is the subgraph of $U_r(v,G)$ whose edges are those of $U_r(v,G)$ not
in $T(v)$). Note that $R$ only consists of unexamined (at this point)
edges and rejected edges.

This model with $R$ and $T$ defined 
as above will be used in Sections \ref{sec:f_5} and \ref{sec:f_4}.

\subsection{Reveal vertices, then connect them}
\label{subsec:reveal}
Let $r > 0$ and $v \in V$. This model consists of two phases:
the creation phase determines the vertices of $U_r(v,G)$ while the
connection phase gradually reveals all edges of $U_r(v,G)$,
separating
it to a spanning tree $T$ and a subgraph $R$ containing all other
edges.
\begin{description}
\label{subsec:gradual}
\item {\bf{Creation phase} }
This phase constructs $L_{i+1}$ given $L_i$ (starting at $i=0$ 
and ending at $i=r-1$) in the following manner:
for every $u \notin L_{\leq i}$, flip a coin with probability $p$ a 
total of $\vert L_{i} \vert$ times or until the first "yes" answer, 
whichever comes first. In case of "yes" add $u$ to $L_{i+1}$.
\item {\bf{Connection phase} }
Connect $L_i$ to $L_{i-1}$, starting at $i=r$ and ending at $i=1$. 
The connection of $L_i$ to $L_{i-1}$ consists of two steps:
	\begin{description}
	\item {\bf{Inner step} } Connect every couple of vertices in 
		$L_i$ randomly and independently with probability $p$.  
	\item {\bf{Counting step} }
	For every $u \in L_i$, let $k_u \leq \vert L_{i-1} \vert$ be the 
       	number of coin flips taken until the first "yes" determined 
       	that $u$ is in $L_i$ in the creation phase. 
	Flip the coin $\vert L_{i-1} \vert - k_u$ more times. Let 
       	$t_u \geq 0$ be the number of additional "yes" answers obtained.
	\item {\bf{Linkage step}}
	For every $u \in L_i$, reveal the neighbours of $u$ in $L_{i-1}$: choose a vertex in $L_{i-1}$ randomly. 
	Connect it to $u$ and add this edge to $T$.
	Now choose (randomly and indpendently) $t_u$ more vertices from $L_{i-1}$, 
	connect each of them to $u$ and add the resulting edges to $R$.  
	\end{description}
\end{description}
All other possible edges can be later examined in an arbitrary order.
This model will be used in Section \ref{sec:f_3}.

%%%%%%%%%%%%%%%%%%%%%%%%%%%%%%%%%%%%%%%%%%%%%%%%%%%%%%%%%%%%%%%%%%%%%%

\section{$4$-Degeneracy and Upper Bound For $f_5(\ell,r)$}

\label{sec:f_5}

%%%%%%%%%%%%%%%%%%%%%%%%%%%%%%%%%%%%%%%%%%%%%%%%%%%%%%%%%%%%%%%%%%%%%%
\begin{thm}
\label{thm:main_result_f_5}
Let $r>0$. There exists $\ell_0 = \ell_0(r)$ such that for every
$\ell > \ell_0$:
\begin{equation}
f_5\left(\ell,r\right)<\left(10\ell\log \ell\right)^{r+1}
\end{equation}
\end{thm}

\begin{proof}
Define $d(\ell) := 3\ell\log{\ell}$. Our choice of a random graph for
the proof is based on the following proposition.
\begin{prop}
\label{prop:not_l_colorable}
Any random graph $G_{n,p}$ with $np = d(\ell)$ satisfies w.h.p.\@
\begin{equation}
\chi\left(G\right)>\ell
\end{equation}
\end{prop}

\begin{proof}
By a standard first moment argument (see \cite{BollobasErdos1975}), w.h.p.\@ there is no independent
set of size $(1+o(1))\frac{2\log{np}}{p} =
(1+o(1))\frac{2\log{\ell}}{p}$ in $G$. Consequently, 
\begin{equation}
\chi(G) \geq (1-o(1))\frac{n}{\frac{2 \log{\ell}}{p}} = (1-o(1))
\frac{d}{2\log{\ell}} = (1-o(1)) \frac{3 \ell
\log{\ell}}{2\log{\ell}} > \ell
\end{equation}
for $\ell$ large enough.
\end{proof}
Take the random graph $G = (V, E) = G_{n, p}$ with $n = (10\ell
\log{\ell})^{r+1}$ and $p = \frac{3}{10}(10\ell \log{\ell})^{-r}$. 
$G$ is not $\ell$-colorable with high probability since $np=d(\ell)$.
We will show that w.h.p.\@ every $r$-ball in $G$ is $4$-degenerate
(and hence $5$-colorable).
\begin{lem}
\label{lem:max_deg}
Fix $r > 0$ and let $\epsilon > 0$ be an arbitrary constant. The
maximum degree of a vertex in the random graph $G_{n,p}$ with $n =
(10 \ell \log{\ell})^{r+1}$ and $p = \frac{3}{10} (10 \ell
\log{\ell})^{-r}$ is w.h.p.\@ no more than $(1+\epsilon)d$. 
\end{lem}
\begin{proof}
Let $v \in V$. We have $deg(v) \sim Bin(n-1, p)$ and $\mu = E[deg(v)]
= d - p$. 
We use the following known Chernoff bound (see A.1.12 in 
\cite{AlonSpencer08}): For a binomial random 
variable $X$ with expectation $\mu$,  and for all $\epsilon >0$ (including
$\epsilon >1$):
\begin{equation}
\label{eq:cher_bound1}
\Pr(X > (1+\epsilon)\mu) 
< \left( \frac{e^{\epsilon}}{(1+\epsilon)^{(1+\epsilon)}}
\right)^{\mu} 
\end{equation}
Noting that $(1+\epsilon)d > (1+\epsilon)\mu$, this bound in 
our case implies
\begin{align}
\Pr \left[ deg(v) \geq (1+\epsilon) d \right] 
<
\left[ \frac{e^\epsilon}{(1+\epsilon)^{(1+\epsilon)}} \right]^{\mu}
=
{\gamma_\epsilon}^{d-p} 
\end{align}
Where $\gamma_\epsilon =  {e^\epsilon}{(1+\epsilon)^{-(1+\epsilon)}}
< 1$ is a positive constant.
Therefore, the probability that there exists a vertex with degree
$\geq (1+\epsilon)d$ is no more than
\begin{align}
\begin{split}
n {\gamma_\epsilon}^{d-p} 
=& e ^ {\log{n} + (d-p) \log{\gamma_\epsilon}}
=   e ^ {(1+o(1))(r+1) \log{\ell} - (1+o(1))
\log(1/{\gamma_\epsilon}) \cdot 3 \ell \log{\ell}} \\
\leq& \ell ^ {(2+o(1))r - (3+o(1))\log(1/{\gamma_\epsilon}) \ell}
\xrightarrow{\ell \rightarrow \infty} 0
\end{split}
\end{align}
Hence with high probability the maximum degree is $<  (1+\epsilon)
d$.
\end{proof}
\begin{lem}
\label{lem:dist_r}
Fix $r$ and let $\epsilon > 0$. Then with high probability all
$r$-balls in $G_{n,p}$ (with $n,p$ as before) contain at most
$(1+\epsilon)^r d^r$ vertices.
\end{lem}
\begin{proof}
The max degree in the graph is w.h.p.\@ $<(1+\epsilon)d$. In this
case,

an easy induction shows that every $i$-ball in the graph has  at most
$(1+\epsilon)^i d^i$ vertices.

Setting $i=r$ gives the desired result.
\end{proof}
We are now ready to prove the main result of this section.
\begin{thm}
\label{thm:4_degenerate}
Fix $r$ and let $n = (10 \ell \log{\ell})^{r+1}$, $p = \frac{3}{10}
(10 \ell \log{\ell})^{-r}$. Then with high probability, every
$r$-ball in $G_{n,p}$ is $4$-degenerate.
\end{thm}

To prove this, note that the probability that not every $r$-ball is
$4$-degenerate is no more than
\begin{align}
\begin{split}
\nonumber
 \Pr \left[ \exists v : U_r(v,G) \mbox{ not $4$-degenerate and
}\forall u \in V : \deg(u) < (1+\epsilon)d \right] 
& + \\
+ \Pr \big[ \exists u \in V : \deg(u) \geq (1+\epsilon)d \big] 
& \leq  \\
\Pr \left[ \exists v : U_r(v,G) \mbox{ not $4$-degenerate} \bigg\vert
\forall u \in V : \deg(u) < (1+\epsilon)d \right] + o(1) 
&\leq  \\
n \Pr \left[ U_r(v_0,G) \mbox{ not $4$-degenerate} \bigg\vert \forall
u \in V : \deg(u) < (1+\epsilon)d \right] + o(1) &
\end{split}
\end{align}
Where $v_0 \in V$ is an arbitrary vertex.
It is therefore enough to show that for fixed $r>0, v \in V$ and
suitable $\epsilon > 0$:
\begin{equation}
\label{eq:enough_to_show}
\lim_{\ell \rightarrow \infty}n \Pr \left[ U_r(v,G) \mbox{ not
$4$-degenerate} \bigg\vert \forall u \in V : \deg(u) < (1+\epsilon)d
\right] = 0
\end{equation}
For the rest of the proof, assume that the maximum degree of $G$ 
is less than $(1+\epsilon)d$. Fix $v \in V$.

A non-$4$-degenerate $r$-ball contains a subgraph with average degree
at
least $5$, hence it is enough to show that with probability high
enough,
every subgraph $S = (V_S, E_S) \subseteq U_r(v,G)$ satisfies $|E_S| <
5|V_S| / 2$.

Construct a spanning tree $T$ with root $v \in V$ for $U_r(v,G)$
in the spanning tree model described in Subsection \ref{subsec:spantree}.

Let $S = (V_S, E_S) \subseteq U_r(v,G)$ be an induced subgraph and
put $s = |V_S|$.  Assume that $s \geq 6$ (as every subgraph on $< 6$
vertices has minimal degree $\leq 4$). 

The possible edges of $S$ are either in $T$ or rejected from the
graph
or not examined yet. $S \cap T$ is a forest and contains at most
$s-1$ edges. 
$S \setminus T$ contains at most 
$\binom{s}{2}$ 
unexamined possible
edges (all other edges are rejected). The probability that an
unexamined
edge is accepted to the graph is no more than $p$. Note that here
we ignore the conditioning on the maximum degree. By the FKG
Inequality (c.f., e.g., \cite{AlonSpencer08}, Chapter 6) this
conditioning can only reduce the probability that we are bounding.
Let $X$ be the random
variable that counts the number of edges in $S \setminus T$. Then $X$
is
dominated by $Bin \left(\binom{s}{2}, p \right)$. That is, for a
random
variable $Y \sim Bin \left(\binom{s}{2}, p \right)$ we have $\Pr(X >
k) \leq \Pr(Y > k)$ for every $k$.
Hence
\begin{equation}
\Pr\left( |E_S| \geq \frac{5s}{2}\right) 
\leq 
\Pr\left(X >  \frac{3s}{2}\right) 
\leq 
\Pr\left(Y >  \frac{3s}{2}\right)
\end{equation}
The expectation of $Y$ is $\mu = \binom{s}{2}p$. 
An easy consequence on the Chernoff bound in \eqref{eq:cher_bound1} 
implies that 
\begin{equation}
\Pr\left(Y > (1+\tau)\mu \right) < \left( \frac{e}{1+\tau} \right)^{(1+\tau)\mu}
\end{equation}
Putting $1+\tau=\frac{3}{p(s-1)}$ we get
\begin{equation}
\Pr\left(Y >  \frac{3s}{2}\right) < \left( \frac {e p (s-1)} {3}
\right)^{3s / 2} < (ps)^{3s/2}
\end{equation}
Pick $\epsilon = \frac{1}{9}$. The number of induced subgraphs $S
\subseteq U_{r}\left(u,G\right)$ on $s$ vertices is
\begin{equation}
\binom{\vert{U_r}(u,G)\vert}{s} \leq \binom{(1+\epsilon)^r d^r}{s}
\leq e^s [(1+\epsilon)d]^{rs} s^{-s} = e^s \left[\frac{10d}{9}
\right]^{rs} s^{-s}
\end{equation}
The probability that $U_r(v,G)$ is not $4$-degenerate is therefore no
more than
\begin{align}
\sum_{s=6}^{[(1+\epsilon)d]^r} 
{  e^s \left[\frac{10d}{9} \right]^{rs}  s^{-s} (ps)^{3s/2} } 
= \label{eq:after_setting_values}
\sum_{s=6}^{[(1+\epsilon)d]^r}
\left[ ep \left(\frac{10}{9}d \right)^r \right]^s  \left[ ps
\right]^{s/2}
\end{align}
but $ep(\frac{10d}{9})^r = \frac{3e}{10} (\frac{10d}{3})^{-r}
(\frac{10d}{9})^r = \frac{3e}{10} 3^{-r} < 1/3$, and the last
expression is
\begin{align}
& \leq 
\sum_{s=6}^{[(1+\epsilon)d]^r}
3^{-s}  \left( ps \right)^{s/2}
\leq 
\sum_{s=6}^{d^{1/10}}
\left( ps \right)^{s/2}
+
\sum_{s=d^{1/10}+1}^{\infty}
3^{-s} \\
& \leq  d^{1/10} (pd^{1/10})^3 + 3^{-d^{1/10}} \leq d^{-26r/10} +
3^{-d^{1/10}}
\end{align}
Since $n = (10 \ell \log{\ell})^{r+1} \leq (4d)^{r+1} \leq 
(4d)^{2r}$, we conclude that
\begin{align}
n \Pr \left[ U_r(v,G) \mbox{ not $4$-degenerate} \bigg\vert \forall u
\in V : \deg(u) < (1+\epsilon)d \right] \leq& \\
\left[d^{-26r/10} +  3^{-d^{1/10}}\right] (4d)^{2r} \leq 
 O(1) \left[d^{-r/2} + e^{-d^{1/10} + 2r \log{d}}\right]
\xrightarrow{\ell \rightarrow \infty} 0 &
\end{align}
This proves \eqref{eq:enough_to_show} and completes the proof of the
Theorem.

Theorem \ref{thm:main_result_f_5} follows from
\ref{prop:not_l_colorable} and the last Theorem.
\end{proof}

%%%%%%%%%%%%%%%%%%%5%%%%%%%%%%555

\section{Upper Bound For $f_4(\ell,r)$}

\label{sec:f_4}

%%%%%%%%%%%%%%%%%%%%%%5%%%%%%%%5
\begin{thm}
\label{thm:bound_on_f_4}
Let $r > 0$. There exists $\ell_0(r)$ such that for every $\ell >
\ell_0$:
\begin{equation}
f_4(\ell,r) <  (10\ell \log{\ell}) ^ {r+1}
\end{equation}  
\end{thm}

\begin{proof}
Once again we take the random graph $G_{n,p}$ with $n = (10\ell
\log{\ell}) ^ {r+1}, p = \frac{3}{10}(10\ell \log{\ell}) ^ {-r}$ and
assume that the maximum degree in $G$ is less than $(1+\epsilon)d =
10d/9$ (taking $\epsilon = 1/9)$.

Let $v \in V$ and construct a spanning tree $T(v)$ for $U_r(v,G)$ as
in Subsection \ref{subsec:spantree}. Let $R(v)  = U_r(v,G) \setminus
T(v)$ be the subgraph of all other edges of $U_r(v,G)$. At this
point, the possible edges of $R$ are either rejected or unexamined.

Suppose that $R$ is $2$-colorable. 
$T$ is a tree and is thus $2$-colorable. The cartesian multiple of a
$2$-coloring of $T$ and a $2$-coloring of $R$ is a valid $4$-coloring
of $U_r(v,G) = T \cup R$.

To make $R$ 2-colorable, it is enough to get rid of all cycles of odd
length in it. This can be done by deleting a vertex (or an edge) from
each such cycle. 
The expected number of cycles of length $k$ in $R(v)$ is no more than
$$
\binom{(1+\epsilon)^r d^r }{k} \frac{(k-1)!}{2} p^k \leq 
\frac{(1 + \epsilon)^{rk}d^{rk} p^k}{2k} 
$$
\begin{equation}
\label{eq:num_cycles_length_k}
\leq \frac{1}{2k} \left( \frac{10d}{9} \right)^{rk} 
\left( \frac{10d}{3} \right)^{-rk} \leq
\frac{1}{2k} 3^{-k}
\end{equation}
Consequently, the expected number of 
cycles (in particular, of odd cycles) in $R(v)$ is bounded by
\begin{equation}
\label{eq:odd_length_cycles}
\sum_{i = 1}^{\infty} \frac{1}{2(2i+1)} 3^{-2i-1} < \frac{1}{100}
\end{equation}
And so the probability that $R(v)$ is not $2$-colorable is less then $1/100$.

Let $G'$ be a graph obtained from $G$ by removing every $v$ for which $R(v)$ contains an odd cycle (that is, the center of each $r$-ball for which $R$ is not $2$-colorable). Observe that $\ell\chi_r(G') \leq 4$. By
\eqref{eq:odd_length_cycles}, the expected number of vertices that
need
to be removed to obtain $G'$  is less than $\frac{n}{100}$. By
Markov's
inequality, with probability at least $1/2$ the number of vertices to
be removed is less than $\frac{n}{50}$ (note that this computation is without the conditioning on the maximum degree, but by the FKG inequality the same estimate holds
also after this conditioning). 

On the other side, w.h.p.\@ there is no independent set of size
$(1+o(1)) \frac{2\log(d)}{p}$ in $G$ (as was discussed in the proof
of
\ref{prop:not_l_colorable}). Consequently there is no independent set
of
such size in $G'$. We conclude that with probability $\geq
\frac{1}{2}
- o(1)$, the chromatic number of $G'$ is at least
\begin{equation}
\frac{n - \frac{n}{50}}{(1+o(1)) \frac{2\log(d)}{p}} 
= (1-o(1)) \frac{49d}{100\log{d}} 
= (1-o(1)) \frac{49 \cdot 3\ell \log{\ell}}{100 \log{\ell}} > \ell
\end{equation}
For $\ell$ large enough. Recall that these estimates are only true
assuming the maximum degree is $< (1+\epsilon)d$, but this property
holds with high probability. 

Thus, the process described above generates with probability
$\frac{1}{2} - o(1)$ a graph $G'$ on at most $(10\ell
\log{\ell})^{r+1}$
vertices which is not $\ell$-colorable, but with $r$-local chromatic
number $\leq 4$. This completes the proof.
\end{proof}

%%%%%%%%%%%%%%%%%%%5%%%

\section{$2$-Degeneracy And Upper bound For $f_3(\ell,r)$}

\label{sec:f_3}

%%%%%%%%%%%%%%%%%%%%%%5
The main result proved in this section is
\begin{thm}
\label{thm:bound_on_f_3}
Let $r > 0$. There exists $\ell_0(r)$ such that for every $\ell >
\ell_0$:
\begin{equation}
f_3(\ell,r) <  (10\ell \log{\ell}) ^ {r+1}
\end{equation}  
\end{thm}
To prove this, we show the following.
\begin{thm}
\label{thm:prob_2_deg}
Let $r > 0$, $v \in V$ where $G = G_{n,p} = (V,E)$, $n = (10 \ell \log{\ell})^{r+1}$,
$p = \frac{3}{10} (10 \ell \log{\ell})^{-r}$. Then
$U_r(v,G)$ is $2$-degenerate with probability at least $0.99 - o(1)$.
\end{thm}
The rest of this section is designed as follows. First it is shown that 
Theorem \ref{thm:bound_on_f_3} follows easily from Theorem \ref{thm:prob_2_deg}. 
To prove \ref{thm:prob_2_deg}, 
we consider an algorithm that checks if $U_r(v,G)$ is $2$-degenerate 
while revealing it as in Subsection \ref{subsec:reveal}. 
The algorithm is shown to be valid (that is, a "yes" answer implies that $U_r(v,G)$ is indeed $2$-degenerate). 
The last part of this section shows that a "yes" answer is returned 
with probability $> 0.99 - o(1)$.

To see why \ref{thm:bound_on_f_3} follows from
\ref{thm:prob_2_deg}, note that the expected number of non-$2$-degenerate
$r$-balls in $G_{n,p}$ is no more than $(\frac{1}{100} + o(1))n$.
Taking $G = G_{n,p}$ and deleting the centers of all
non-$2$-degenerate $r$-balls generates a graph $G'$ with
$\ell\chi_r(G') \leq 3$. Markov's inequality implies, 
as in Section \ref{sec:f_4}, that with probability at least $1/2-o(1)$ we do not delete
more than $\frac{2}{100} n$ centres, thus 
$\chi(G') > \ell$ holds with probability $> \frac{1}{2} - o(1)$. This
completes the proof of Theorem \ref{thm:bound_on_f_3}. 

The rest of this section is dedicated to proving Theorem \ref{thm:prob_2_deg}.
Let $v \in V$. For the (more complicated) analysis of this problem,
we use the model of revealing $U_r(v,G)$ presented in subsection
\ref{subsec:reveal}.

We start with some definitions.
First, recall the definition of a level with respect to a vertex.
\begin{definition}
For a subgraph $F=(V_F, E_F) \subseteq U_r(v,G)$, let 
\[
L_i(v, F) = \{ u \in V_F  : d(u,v) = i \}
\]
denote the \emph{$i$-th level} (with respect to $v$ in $F$). 
Moreover, define 
\[
L_{\geq i}(v,F) = \bigcup_{j=i}^{r} { L_j (v, F) } \mbox{ ; } L_{\leq
i}(v,F) = \bigcup_{j=0}^{i} { L_j (v, F) }
\]
\end{definition}
Note that the distance $d(u,v)$ here denotes distance in $G$, not in $F$.

The notation $L_i$  (without specifying $v$ and $F$) refers to
$L_i(v,G)$. The same holds for $L_{\geq i} = L_{\geq i}(v,G)$ and
$L_{\leq i} = L_{\leq i}(v,G)$. For convenience we will
also sometimes use these notations to describe the induced subgraph
of $F$ on the relevant set of vertices.

The next definition presents a few special types of paths and cycles, to be used later when describing and analyzing the algorithm. 
\begin{definition} 
Let $F \subseteq U_r(v,G)$.
\begin{itemize}
\item An \emph{$i$-path} in $F$ is a simple path in $L_{\geq i}(v,
F)$ whose endpoints belong to $L_i (v,F)$.
\item An $i$-cycle in $F$ is a simple cycle in
$L_{\geq i}(v,F)$ with at least one vertex in $L_i (v,F)$.
\item An $i$-\emph{horseshoe} in $F$ is a path of the form
\[
u w_1 \ldots w_k z
\]
 where $u,z \in L_{i-1}(v, F)$, $k \geq 1$, $u w_1, w_k z \in R$ and $w_1 \ldots w_k $ is an $i$-path in $F$.
Specifically in the case $k=1$ we also require $u \neq z$.
\item An \emph{$i$-sub-horseshoe} in $F$ is a path of the form
\begin{equation}
u' w_1 \ldots w_k z'
\end{equation}
where $u',z' \in L_{i-1}(v,F)$, $k \geq 1$, $u'w_1, w_k z' \in F$ 
and $w_1 \ldots w_k$ is included in the interior of some $i$-horseshoe.
Specifically in the case $k=1$ we also require $u' \neq z'$.
\end{itemize}
\end{definition}
Note that every $i$-horseshoe is also an $i$-sub-horseshoe, but the
other direction is not true in general.
Here the interior of a path denotes the induced subpath on all vertices except for the endpoints. 

\subsection{Algorithm for checking if $U_r(v,G)$ is $2$-degenerate}
\label{subsec:algo_2_deg}
Consider the following algorithm to check if $U_r(v,G)$ is
$2$-degenerate. This algorithm always returns "no" if the ball is not
$2$-degenerate, but is not assured to return "yes" for a
$2$-degenerate ball.
We will show that the probability of a "yes" answer is high enough,
implying that the $r$-ball is $2$-degenerate with high enough
probability. 

Our algorithm (applied while revealing $U_r(v,G)$ as described in
Subsection \ref{subsec:reveal}) maintains a subgraph $F$ which initially
consists of all vertices of $U_r(v,G)$ where the edges are not yet
revealed. It then gradually reveals information about the edges of
$U_r(v,G)$ and adds these edges to $F$ while deleting vertices whose
neighbours in $F$ are revealed but their degree is at most $2$.
Some conditions might lead to a "no"
answer returned by the algorithm, but if it succeeds to delete all
vertices of $F$, it returns "yes".

It can be seen as a pessimistic version of the naive approach of
trying to remove vertices of degree $\leq 2$ from the graph until all
the vertices are removed (a "yes" answer) or until a subgraph with
minimum degree $\geq 3$ is revealed (a "no" answer). Our algorithm is
less accurate but easier to analyze than the naive approach.

\subsubsection*{Algorithm \ref{subsec:algo_2_deg} - detailed
description}
\label{subsub:algorithm}
\begin{enumerate}
\item\label{enum:creation} Creation phase
	\begin{enumerate} 
	\item Reveal the levels $L_i$ of $U_r(v,G)$.
		\begin{enumerate}
		\item\label{enum:creation_ret_no} If for some $1\leq i \leq r$ it
holds that $|L_i| > (1+\epsilon) d |L_{i-1}|$ with $\epsilon = 1/9$, return "no".
		\item\label{enum:init_f} Initialize a subgraph $F$ with all vertices
of $U_r(v,G)$ and no edges.
		\end{enumerate}	
	\end{enumerate}
\item\label{enum:conn} Connection phase: For every level $L_i$ from
$i=r$ to $i=1$ do:
	\begin{enumerate}
	\item\label{enum:edges_i} Inner step: reveal all inner edges of
$L_i$, i.e.\@ edges in $G$ of the form $\{ u,u' \}$ where $u \neq u'
\in L_i$. Add them to $F$.
		\begin{enumerate}
		\item\label{enum:cycle1} At this point all edges of $L_{\geq
i}(v,F)$ are revealed. If there exists an $i$-cycle in $F$, return
"no".
		\end{enumerate}
	\item\label{enum:counting} Counting step: for every $u \in L_i$,
determine how many neighbours it has in $L_{i-1}$.
		\begin{enumerate}
		\item\label{enum:remove}  At this point we know the degree (in $F$)
of all vertices in $L_{\geq i}$. If there exists $u \in L_{\geq
i}(v,F)$ with degree $\leq 2$ in $F$ - delete $u$. Repeat until all
vertices of $L_{\geq i}(v,F)$ are of degree $> 2$ in $F$. 		
		\item\label{enum:horseshoes} The number of $i$-sub-horseshoes in
$F$ is also known now. If this number is bigger than $b_i$ (to be
determined later), return "no". Moreover, the structures of the
$i$-(sub-)horseshoes are known aside from the identities of their
endpoints in $L_{i-1}$.
		\end{enumerate}
	\item\label{enum:link} Linkage step: For every $u \in L_i$, reveal the 
neighbours of $u$ in $L_{i-1}$, adding one of the new edges to $T$ and the others to $R$. 
Add all new edges to $F$. 
		\begin{enumerate}
		\item At this point, all the $i$-horseshoes and $i$-sub-horseshoes are revealed.	
		\end{enumerate}
	\end{enumerate}
\end{enumerate}
Finally, if the connection phase ends without returning "no", the algorithm return "yes".

\begin{lem}[validity of the algorithm]
\label{lem:validity}
If algorithm \ref{subsub:algorithm} returns "yes", then $U_r(v,G)$ is
$2$-degenerate.
\end{lem}
\begin{proof}
Assume that the algorithm returned "yes". In the end of the iteration
$i=1$, $L_{\geq 1}(v,F)$ does not contain cycles - since a "no" has
not been returned before then.  Therefore, $L_{\geq 1}(v,F) = F
\setminus \{ v \}$ is a forest and thus $1$-degenerate, implying that
$F$ is $2$-degenerate at that point. 
Note that the algorithm does not need to inspect the edges between 
$v$ and $L_1$, since the $1$-degeneracy of $F \setminus \{ v \}$ suffices.

Observe that if a vertex $v$ has degree $\leq 2$ in a graph $H$, then
$H$ is $2$-degenerate if and only if $H \setminus \{ v \}$ is
$2$-degenerate.

Let $v_1, \ldots, v_m$ be the ordered sequence of vertices that were deleted from $F$
during the algorithm. Let $F_i = U_r(v,G) \setminus \{ v_1, \ldots, v_i \}$ for $i=0,\ldots,m$. 
Clearly, $v_{i+1}$  is of degree at most $2$ in $F_i$ (since we only delete a vertex if it is of degree 
at most $2$ in $F$ at that point). 
The previous observation implies that $F_i$ is $2$-degenerate if and only if $F_{i+1}$ is $2$-degenerate. 
Moreover, the first argument states that $F_m$ is $2$-degenerate. 
Therefore, by induction $F_i$ is $2$-degenerate for every $i$. 
Noting that $F_0 =U_r(v,G)$ finishes the proof.
\end{proof}

\subsection{Analysis of the algorithm}

We first present notation that is used throughout the analysis.
Afterwards we characterize the set of vertices in $L_{\geq j}$ that
survive iteration $i=j$. We use this characterization to give bounds
(valid with high probability) on the number of $j$-sub-horseshoes
revealed in a given iteration as well as the probability to reveal a
$j$-cycle. This gives us the desired lower bound on the probability
that the algorithm returns "yes", which implies (along with Lemma
\ref{lem:validity}) that an $r$-ball in $G_{n,p}$ is $2$-degenerate
with sufficiently high probability.  

\paragraph{Notation}
The following quantities are of interest for analysing algorithm
\ref{subsub:algorithm}:
\begin{description}
\item[$n_j$] number of vertices in $L_j(v,G)$.
\item[$c_j$] number of $j$-cycles in $F$ at the end of the inner step
\eqref{enum:edges_i} in iteration $i=j$ of the connection phase of
algorithm \ref{subsub:algorithm}.
\item[$h_j$] number of $j$-sub-horseshoes in $F$ at the end of the 
counting step \eqref{enum:counting} in iteration $i=j$ 
of the connection phase.
\end{description}
The next group of notations refers to the probability to get a "no"
answer at some point of the algorithm assuming a "no" has not been
returned before then.
\begin{description}
\item [$q^l$] probability that step \eqref{enum:creation_ret_no}
reveals that $n_{j+1} > (1+\epsilon) d n_j$ for some $j$.
\item[$q_j^c$] probability that $c_j > 0$ assuming the algorithm
has not returned "no" before iteration $i=j$ of the connection phase.
\item[$q_j^h$] probability that $h_j > b_j$ ($b_j$ will be
determined later) assuming the algorithm has not returned "no" before
iteration $i=j$ of the connection phase.
\end{description}
Note that $h_1 = 0$ and these three conditions are the only ones that lead to a
"no" answer, implying the following lemma.
\begin{lem}
\label{lem:prob_alg_no}
The probability that algorithm \ref{subsec:algo_2_deg} returns "no"
is no more than 
\begin{equation}
q^l + \sum_{j=1}^{r} q_j^c + \sum_{j=2}^{r} q_j^h
\end{equation}
\end{lem}

The proof of Theorem \ref{thm:prob_2_deg} follows from the next
Theorem, along with Lemmas \ref{lem:validity} and
\ref{lem:prob_alg_no}.
\begin{thm}
\label{thm:f_3_probs}
The following holds with respect to algorithm \ref{subsec:algo_2_deg}
on $G_{n,p}$ and $v$ defined as above:
\begin{align}
\label{eq:q_l_proof}
q^l = o(1) &
\\
\label{eq:p_r_c_proof}
q_r^c < \frac{1}{100} &
\\
\label{eq:proof_all_other_probs}
\sum_{j=1}^{r-1} q_j^c + \sum_{j=2}^{r} q_j^h = o(1)&
\end{align}
\end{thm}

\begin{proof}
\eqref{eq:q_l_proof} is immediate from Lemma \ref{lem:max_deg}. 

As in \eqref{eq:num_cycles_length_k} and \eqref{eq:odd_length_cycles} and
since $L_r$ is of size at most $(1+\epsilon)^r d^r$, 
the expected number of cycles in $L_r$ is no more than
\begin{equation}
\sum_{k=3}^{\infty} \frac{1}{2k} 3^{-k} < \frac{1}{100}
\end{equation}
which proves \eqref{eq:p_r_c_proof}.
In the rest of the proof we establish \eqref{eq:proof_all_other_probs}.
\paragraph{Horseshoes and sub-horseshoes}
We start by explaining why horseshoes and sub-horseshoes are
important for the analysis of this problem.
\begin{lem}
A vertex in $L_{\geq j}$ might remain in $F$ after step
\ref{enum:remove} of iteration $i = j$ of the algorithm only if it lies
in some $j$-horseshoe of $F$ at that point. 
\end{lem}
\begin{proof}
Observe $F$ at the end of step \ref{enum:counting} in iteration $i =
j$ of the algorithm.
Let $w \in L_{\geq j}(v,F)$ be a vertex that is not contained in any
$j$-horseshoe at this point.
Then there is at most one edge $e$ touching $w$ that is the first edge
of a path $P$ from $w$ to $L_{j-1}$ whose interior is in $L_{\geq j}$
and last edge is in $R$ (note that this interior might also be 
empty if $P$ is a single edge). Otherwise, let $e_1 \neq e_2$ be such edges
and let $P_1, P_2$ be the corresponding paths.
Since $L_{\geq j}(v,F)$ does not contain cycles at this point, 
the interiors of $P_1$ and $P_2$ are disjoint. Thus $w$ lies in
the horseshoe $P_1 \cup P_2$, a contradiction.
Hence there exists at most one edge $e$ of this type. 
We can assume that there exists exactly one.

Let $S_w$ be the connected component of $w$ in $L_{\geq j}(v,F)
\setminus \{e\}$ at this point. Any vertex aside from $w$ has at most
one neighbour in $F$ outside $S_w$ (that is its parent in $T$).
Moreover, $S_w$ is a forest and thus contains a leaf $z \neq w$. 
$z$ has degree $\leq 2$ in $F$ and can be removed from it. \\
This process ends when all vertices of $S_w \setminus \{w\}$
are removed from $F$, leaving $w$ with at most two neighbours:
its parent in $T$ and the other enndpoint of $e$. At this point,
$w$ can be removed from $F$, completing the proof.
\end{proof}
One can check that the following is a consequence of the last lemma 
providing a similar result for edges.
\begin{cor}
An edge of $U_r(v,G)$ that has an endpoint in $L_{\geq j}$ might remain 
in $F$ after step \ref{enum:remove} of iteration $i = j$ of 
the algorithm only if it lies in some $j$-sub-horseshoe of $F$ at that point. 
\end{cor}
Recall that the bounds $b_j$ in step \ref{enum:horseshoes} have not
been defined yet.
Take $b_1 = 0$ since there are no $1$-horseshoes. For $1 < j \leq r$
take $b_j = \frac{n_{j-1}}{\ell}$.
The reasoning for these choices will be clearer later.
\paragraph{$r$-horseshoes and $q_r^h$}
A $r$-horseshoe of length $k+1$ is a path in $R$ with both endpoints
in $L_{r-1}$ and $k > 0$ interior points in $L_r$. The number of
candidates to be $r$-horseshoes of length $k+1$ is $\leq n_{r-1}^2
n_r^k$. FKG inequality implies that each candidate is indeed 
a $r$-horseshoe in $U_r(v,G)$ with
probability at most $p^{k+1}$. Such a horseshoe, if exists, forms no
more than $3k^2$ $r$-sub-horseshoes. Combining everything we get
\begin{align}
\label{eq:E[h_r]}
E[h_r| \mbox{algorithm did not return "no" before sampling $h_r$}]
& \leq 
\\
\label{eq:pn_j}
\sum_{k=1}^{\infty} 3k^2 p^{k+1} n_{r-1}^2 n_r^k 
= 
3n_{r-1} (p n_{r-1}) \sum_{k=1}^{\infty} k^2 (pn_r)^k 
& \leq \\
3n_{r-1} \frac{3^{-r}} {d} \sum_{k=1}^{\infty} k^2 3^{-rk}
& \leq
O(1) \frac{n_{r-1}} {d}
\end{align} 
the inequality in \eqref{eq:pn_j} is true since 
\begin{equation}
pn_j \leq \frac{3}{10} \left(\frac{10}{3}d \right)^{-r} (10/9)^j d^j
< 3^{-r} d^{j-r}
\end{equation}
Applying Markov's inequality to \eqref{eq:E[h_r]} we
get: 
\begin{equation}
\label{eq:p_r_h_final}
q_r^h \leq \frac{O(1) \frac{n_{r-1}} {d}}{b_r} = O
\left(\frac{\ell}{d} \right) = O(1/\log{\ell}) = o(1)
\end{equation}

\subsubsection*{$j$-cycles and $j$-horseshoes for $j<r$}
Assume that the algorithm has not returned "no" in step
\ref{enum:creation_ret_no} or in iterations $r, r-1, \ldots, j+1$ of
the connection phase. In particular, the number of
$(j+1)$-sub-horseshoes in $F$ is at most $b_{j+1}$ and there are no
$(j+1)$-cycles in $F$. At this point in the algorithm, the inner
structures of the $(j+1)$-horseshoes are known, but their endpoints
are not yet determined (as the last possible "no" answer of iteration
$j+1$ of the connection phase comes after the inner structures are
determined but before step \ref{enum:link} is taken).

A $j$-cycle has parameters $m,k$ (with $1 \leq m \leq n_j\ ,\ 0 \leq
k \leq m$) if it consists of exactly $m$ vertices in $L_j$, $k$
internally-disjoint $(j+1)$-sub-horseshoes (the interiors are disjoint
since a $j$-cycle is simple) and $m-k$ inner edges of $L_j$. It is
clear that any $j$-cycle in $F$ can be presented in such a form.

A $j$-horseshoe with parameters $m,k$ ($1 \leq m \leq n_j\ ,\ 0 \leq
k \leq m-1$)  is defined similarly: it consists of $m$ vertices in
$L_j$, $k$ internally-disjoint $(j+1)$-sub-horseshoes, $m-1-k$ inner
edges of $L_j$ and two edges down to $L_{j-1}$ that are in $R$.
Again, any $j$-horseshoe can be presented in this form.

We now bound the expected number of $j$-cycles and $j$-horseshoes. We
do so by estimating the number of such objects with parameters $m,k$
for all possible values of $m,k$.

Fix $a_1, \ldots, a_k, b_1, \ldots, b_k \in V$ (not necessarily
distinct) and internally-disjoint $(j+1)$-sub-horseshoes $H_1, \ldots,
H_k$. The probability that a specific $H_i$ has endpoints $a_i, b_i$
is at most $\frac{1}{\binom{n_j}{2}} \leq \frac{4}{n_j^2}$ (this is
true for $n_j \geq 2$ ; if $n_j = 1$ then there are no
$(j+1)$-sub-horseshoes anyway). These $k$ events are independent (as
per step \ref{enum:link} in the connection phase), and the
probability that all of them occur together is at most $\frac{1}{\binom{n_j}{2}^k} \leq \frac{4^k}{n_j^{2k}}$.

There are no more than $b_{j+1}^k$ possible ordered choices of $(H_1,
\ldots, H_k)$. Therefore, the expected number of ordered sets of $k$
internally-disjoint $(j+1)$-sub-horseshoes with endpoints $(a_1, b_1),
\ldots, (a_k, b_k)$ is no more than
\begin{equation}
\frac{4^k b_{j+1}^k} {n_j^{2k}} \leq \frac{4^k}{\ell^k n_j^k}
\end{equation}

\paragraph{$j$-cycles}
First we bound the expected number of $j$-cycles with parameters
$m,k$ in $F$ after step \ref{enum:edges_i} in iteration $i=j$ of the
connection phase.
Fix $m$ vertices $(v_1, v_2, \ldots, v_m) \in L_j$ and order them
cyclically (there are at most $n_j^m$ such orderings). Now fix $k$
couples of neighbouring vertices $a_i, b_i$ in the chosen cyclic
order (there are $\binom{m}{k} \leq 2^m$ possible choices of
$k$-tuples). The expected number of $k$-tuples of internally-disjoint
$(j+1)$-sub-horseshoes with endpoints $a_i, b_i$ is no more than
$\frac{4^k}{\ell^k n_j^k}$. The probability for any other couple of
neighbours in the cyclic ordering to have an edge between them is $p$
independently of everything else. 
Since the expected multiple of independent random variables is the multiple of their expectations, we get that the expected number of $j$-cycles with parameters $m,k$ is no
more than
\begin{equation}
n_j^m 2^m \frac{4^k}{\ell^k n_j^k} p^{m-k} = (2pn_j)^{m-k} \left(
\frac{8}{\ell} \right)^k \leq \left(\frac{1}{d}\right)^{m-k} \left(
\frac{8}{\ell} \right)^k \leq \left(\frac{8}{\ell}\right)^m
\end{equation}
And the total expected number of $j$-cycles is no more than
\begin{equation}
\sum_{m=1}^{n_j} \sum_{k=0}^{m} \left(\frac{1}{d}\right)^{m-k} \left(
\frac{8}{\ell} \right)^k = \sum_{m=1}^{n_j} \left( \frac{8}{\ell}
\right)^m (1+o(1)) 
= \frac{8}{\ell} (1+o(1)) = o(1)
\end{equation}
In particular we get 
\begin{equation}
\label{eq:p_j_c_final}
q_j^c = O(1/\ell) = o(1)
\end{equation}

\paragraph{$j$-horseshoes}
We bound the expected number of $j$-(sub-)horseshoes with parameters
$m,k$ in $F$ after step \ref{enum:counting} in iteration $i=j$ of the
connection phase.
Fix $m$ linearly ordered vertices $(v_1, v_2, \ldots, v_m) \in L_j$
(there are at most $n_j^m$ such orderings).
Now fix $k$ couples of neighbouring vertices $a_i, b_i$ in the chosen
linear ordering (there are $\binom{m-1}{k} \leq 2^{m-1}$ possible
choices). The expected number of $k$-tuples of internally-disjoint
$(j+1)$-sub-horseshoes with endpoints $a_i, b_i$ is no more than
$\frac{4^k}{\ell^k n_j^k}$.
The probability for any other couple of neighbours in the linear
ordering to have an edge between them is $p$ independently of
everything else.
For a vertex in $L_j$, the expected number of neighbours via $R$ it
has in $L_{j-1}$ is no more than $n_{j-1} p$.
Combining all of the above, the expected number of $j$-horseshoes
with parameters $m,k$ is no more than 
\begin{align}
& n_j^m 2^{m-1} \frac{4^k}{\ell^k n_j^k} p^{m-1-k} (n_{j-1}p)^2 
=  \\
& (2 n_j p)^{m-1-k} (n_j p) (n_{j-1} p) (8/\ell)^k n_{j-1}
\leq \\
& d^{(j-r)(m-1-k)} d^{j-r} d^{j-1-r} (8/\ell)^k n_{j-1}
\leq \\
& d^{-(m-1-k)} d^{-3} (8/\ell)^k n_{j-1}
=
\Theta(\log{\ell})^{k} d^{-m-2} n_{j-1}
\end{align}
Each $j$-horseshoe with such parameters contributes no more than
$3m^2$ $j$-sub-horseshoes,
and the total expected number of $j$-sub-horseshoes in $F$ is at most
\begin{align}
& 3\sum_{m=1}^{n_j} \sum_{k=0}^{m-1} \Theta(\log{\ell})^{k} d^{-m-2}
n_{j-1} m^2
\leq \\
& \sum_{m=1}^{n_j} o(1) \Theta(\log{\ell})^{m+2} d^{-m-2} m^2 n_{j-1}
\leq \Theta(\ell)^{-3} n_{j-1}
\end{align}
By Markov's inequality, 
\begin{equation}
\label{eq:p_j_h_final}
q_j^h \leq \frac{\Theta(\ell)^{-3} n_{j-1}}{b_j} \leq
\Theta(\ell)^{-2} = o(1)
\end{equation}
The proof of \eqref{eq:proof_all_other_probs} is now complete by
\eqref{eq:p_j_c_final}, \eqref{eq:p_j_h_final} and since $r$ is
fixed.
\end{proof}
\begin{remark}
Special care should be taken in proofs of this type to ensure that no
source of randomness is used more than once (that is, to prevent the case when
some information is revealed at some point of the algorithm but is
assumed to be random later on).
In particular, note that the information needed to
determine how many $j$-sub-horseshoes there are does not interfere
with the information needed to know, given all the interiors of
$j$-sub-horseshoes without knowing their endpoints yet, what is the
probability that specific $k$ internally-disjoint $j$-sub-horseshoes
have specific $k$ couples of endpoints. 
\end{remark}

\section{$f_c(\ell,r)$ For Non-Constant $c$}

\label{sec:c_non_const}

%%%%%%%%%%%%%%%%%%%%%%%%%%%%%%%%%%%%%%%%%%%%%%%%%%%%%%%%%%%%%%%%%%%%%%
In the previous sections, $f_c(\ell,r)$ with small fixed $c$
values was considered. In this section our results are extended to
large values of $c$. Take $G$ on $n$ vertices, 
$0.98 (10 \ell \log{\ell})^{r+1} \leq n \leq (10 \ell \log{\ell})^{r+1}$
with $\ell\chi_r(G) = 3$ and with no independent set of size
$(1+o(1))\frac{2\log(d)}{p}$, where, as before, $d=3 \ell \log
\ell$ and $p=\frac{3}{10}(10 \ell \log \ell)^{-r}$. 
Such $G$ exists 
%as an easy consequence of
%Lemma \ref{lem:max_deg} and 
by the results in Section \ref{sec:f_3}. 

Construct the following graph $G_k$:
every vertex in the original $G$ is expanded to a $k$-clique. Two
vertices in $G_k$ are connected if they lie in the same clique or if
the cliques in which they lie were neighbours in $G$.
Every independent set in $G_k$ contains at most one vertex from each
clique, and thus the maximal independent set in $G_k$ is of size
$<(1+o(1))\frac{2\log(d)}{p}$.
There are $kn$ vertices in $G_k$ and thus its chromatic number is
(for $\ell$ large enough)
\begin{equation}
\chi \left( G_k \right) \geq \frac{kn}{(1+o(1))\frac{2\log(d)}{p}} >
k\ell
\end{equation}
Every $r$-ball in $G_k$ is contained in an expanded $r$-ball from
$G$. Thus 
\begin{equation}
\ell{\chi_r}(G_k) \leq 3k
\end{equation}
We conclude that for $\ell^*$ large enough 
\begin{equation}
f_{3k}(k\ell^*, r) < kn \leq k (10 \ell^* \log{\ell^*})^{r+1}
\end{equation}
Taking $c=3k$, $\ell = k\ell^*$ the last result implies that 
\begin{equation}
f_c(\ell, r) < \frac{c}{3} \left( 10 \frac{\ell}{c/3} \log{\left(
\frac{\ell}{c/3} \right)} \right)^{r+1} < \frac{(30 \ell
\log{\ell})^{r+1}}{c^r}
\end{equation}
When $c$ and $\ell$ are not of this form, we need to replace them by
$3 \lfloor c/3 \rfloor \leq c$ and $\lfloor c/3 \rfloor \big\lceil
\frac{\ell}{\lfloor c/3 \rfloor} \big\rceil \geq \ell$ respectively.
The following Theorem summarizes the discussion.
\begin{thm}
\label{thm:large_c}
There exists $\ell_0$ such that for every positive $c$ divisible 
by $3$ and $\ell \geq max(c, \ell_0)$ divisible by $c/3$:
\begin{equation}
f_c(\ell, r) < \frac{(30 \ell \log{\ell})^{r+1}}{c^r}
\end{equation}
Thus for every $c \geq 3$:
\begin{equation}
f_c(\ell, r) < \frac{[O(\ell \log{\ell})]^{r+1}}{c^r}
\end{equation}
\end{thm}

\begin{remark}
The contribution of $c$ in this upper bound is $c^{-r}$, 
whereas this contribution in the corresponding lower bound by 
Bogdanov in \eqref{eq:bogd} is $c^{-r-1}$. 
\end{remark}

%%%%%%%%%%%%%%%%%%%%%%%%%%%%%%%%%%%%%%%%%%%%%%%%%%%%%%%%%%%%%%%%%%%%%%
\section{${f_c}(\ell,1)$}
\label{sec:r_is_1}
%%%%%%%%%%%%%%%%%%%%%%%%%%%%%%%%%%%%%%%%%%%%%%%%%%%%%%%%%%%%%%%%%%%%%%
As stated in Section \ref{sec:bg} it is known that
${f_2}(\ell,1) = \Theta \left( \ell^2 \log \ell \right)$. 
In this section it is shown that $f_c(\ell, 1) = \Theta(\ell^2
\log{\ell})$ for any fixed $c \geq 2$. Since $f_c(\ell,1) \leq
f_2(\ell,1)$, we only need to show that $f_c(\ell,1) = \Omega(\ell^2
\log{\ell})$ for fixed $c \geq 2$.
\begin{thm}
\label{thm:lower_bound}
There exists $\alpha > 0$ such that for every $\ell \geq c \geq 2$ 
\begin{equation}
{f_c}(\ell,1) \geq \alpha \frac{\ell^2 \log \ell}{c \log c}
\end{equation} 
In particular, for any fixed $c \geq 2$:
\begin{equation}
f_c(\ell, 1) = \Theta(\ell^2 \log{\ell})
\end{equation}
\end{thm}

\begin{proof}
Let $G=(V,E)$ be a graph on $n = f_c(\ell, 1) + 1$ vertices
with $\ell{\chi_1}(G) \leq c$ but $\chi(G) > \ell$. Our goal is to
show\footnote{In fact we need to show this for $n-1$ instead of $n$
but it is clearly equivalent.} that $n \geq \alpha \frac{\ell^2 \log
\ell}{c \log c}$ for a suitable choice of $\alpha$. By taking
a critical subgraph of $G$ we can assume that the minimum degree of
$G$ is at least $\ell$ and clearly we can also assume that
$n \leq \zeta \ell^2 \log{\ell}$  for
some absolute constant $\zeta > 0$. By these assumptions, the average
degree $d$ in $G$ satisfies 
$\ell \leq d < n \leq \zeta \ell^2 \log{\ell}$.

\paragraph{Large independent set in $G$}
Observe that
\begin{itemize}
\item There exists $v \in V$ with $\deg(v) \geq d$. The 
neighborhood of $v$ is $c$-colorable, and thus contains 
an independent set of size at least $d/c \geq \ell/c$.
\item The first author \cite{alon1996independence} showed that there 
exists $\beta > 0$ such that any graph $G$ on $n$ vertices with 
average degree $d \geq 1$ and $\ell\chi_1(G) \leq c$ 
contains an independent set of size
\[
\frac {\beta} {\log{c}}   \frac{n}{d} \log{d}
\]
\end{itemize}
\begin{lem}
There exists an independent set of size $\geq \delta \sqrt{
\frac{n\log{n}}{c\log{c}}}$ in $G$ where $\delta > 0$ is a suitable
global constant.
\end{lem}
\begin{proof}
There exists an independent set of size
\begin{equation}
\max \bigg\{ \frac{d}{c}, \frac {\beta} {\log{c}}   \frac{n}{d}
\log{d} \bigg\}
\geq 
\sqrt {\frac{d}{c} \frac { \beta} {\log{c}}   \frac{n}{d} \log{d}  }
\geq 
\sqrt { \frac{\beta n \log{\ell} } {c \log{c}}}
\geq 
\delta \sqrt {\frac{n \log{n}}{c \log{c}}}
\end{equation}
as needed.
\end{proof} 
Removing an independent set of size $\delta \sqrt {\frac{f_c(\ell,1)
\log{f_c(\ell,1)}}{c \log{c}}} \leq \delta \sqrt {\frac{n \log{n}}{c
\log{c}}}$ from $G$ results in a non-$(\ell-1)$-colorable graph.
Hence
\begin{equation}
\label{eq:induct1}
f_c(\ell-1, 1) \leq f_c(\ell,1) - \delta \sqrt {\frac{f_c(\ell, 1)
\log{f_c(\ell, 1)}}{c \log{c}}}
\end{equation}
For $\delta$ small enough and $c \geq 2$, the function 
\begin{equation}
h(x):= x -  \delta \sqrt{ \frac{x\log x}{c\log c}}
\end{equation}
is increasing in the domain $[2, \infty)$. Now take 
$\alpha = \min(1, \delta^2 / 9)$ and fix $c \geq 2$. We will show
that $f_c(\ell,1) \geq \alpha \frac{\ell^2 \log{\ell}} {c \log{c}}$
for every $\ell \geq c$ by induction on $\ell$.
The base case $\ell = c$ satisfies 
\begin{equation}
{f_c}(c,1) = c = \frac{\ell^2 \log \ell}{c \log c} \geq \alpha
\frac{\ell^2 \log \ell}{c \log c}
\end{equation}
Assuming that ${f_c}(\ell,1) \geq \alpha \frac{\ell^2 \log \ell}{c
\log c}$ and using \eqref{eq:induct1} we get
\begin{equation}
\label{eq:h_bigger}
\alpha \frac{\ell^2 \log \ell}{c \log c} \leq {f_c}(\ell+1,1) -
\delta \sqrt{ \frac{{f_c}(\ell+1,1)\log{{f_c}(\ell+1,1)}}{c\log{c}}}
= h\left( {f_c}(\ell+1,1) \right)
\end{equation}
Note that $f_c(\ell+1, 1) \geq \ell+1$. If $\alpha \frac{(\ell+1)^2
\log (\ell+1)}{c \log c} \leq \ell+1$ then we are finished. 
Otherwise, take $x=\alpha \frac{(\ell+1)^2 \log (\ell+1)}{c \log c}
\geq \ell+1 \geq 3$. 
Then
\[
\begin{split}
x - \alpha \frac{\ell^2 \log \ell}{c \log c} &= \alpha \left[
((\ell+1)^2 - \ell^2) \frac{\log (\ell+1)}{c \log c} + (\log(\ell+1)
- \log{\ell}) \frac{\ell^2}{c \log c} \right] \\
&\leq 
\alpha \frac{(2\ell+1) \log (\ell+1) + \ell}{c \log c}
\leq
3\alpha \frac{(\ell+1) \log (\ell+1)}{c \log c} \\
& =
3\alpha \sqrt{ \frac { \frac{(\ell+1)^2 \log(\ell+1)} {c\log{c}}
\log(\ell+1)} {c \log{c}} }
\leq
3\sqrt{\alpha} \sqrt{ \frac {x\log{x}} {c\log{c}}} \leq \delta \sqrt{
\frac {x\log{x}} {c\log{c}}}
\end{split}
\]
The last inequality and \eqref{eq:h_bigger} imply that
\begin{equation}
h(x) \leq  \alpha \frac{\ell^2 \log \ell}{c \log c} \leq
h(f_c(\ell+1,1))	\label{eq:h_smaller}
\end{equation}
By the monotonicity of $h$,
\begin{equation}
\alpha \frac{(\ell+1)^2 \log (\ell+1)}{c \log c} = x \leq
f_c(\ell+1,1) 
\end{equation}
finishing the induction step and completing the proof.
\end{proof}

%%%%%%%%%%%%%%%%%%%%%%%%%%%%%%%%%%%%%%%%%%%%%%%%%%%%%%%%%%%%
\section{Final Remarks}
\label{sec:conc_remarks}
%%%%%%%%%%%%%%%%%%%%%%%%%%%%%%%%%%%%%%%%%%%%%%%%%%%%%%%%%%%%

%%%%%%%%%%%%%%%%%%%%%%%%%%%%%%%%%%%%%%%%%%%%%%%%%%%%%%%%%%%%

\subsection{Non-constant $r$}

%%%%%%%%%%%%%%%%%%%%%%%%%%%%%%%%%%%%%%%%%%%%%%%%%%%%%%%%%%%%
Our bounds for $f_c(\ell, r)$ are  valid for fixed values of $r$.
These
bounds still hold if we require that $r \leq \gamma \ell$ for a
suitable
global constant $\gamma > 0$ instead of requiring $r$ to be fixed.
The
following amendments of the proof need to be made:

\begin{itemize}
\item In Lemma \ref{lem:max_deg} we need to make sure that $\ell ^
{(2+o(1))r - (3+o(1))\log(1/{\gamma_\epsilon}) \ell}
\xrightarrow{\ell \rightarrow \infty} 0$ where $\epsilon = 1/9$ and
$\gamma_\epsilon =  {e^\epsilon}{(1+\epsilon)^{-(1+\epsilon)}} < 1$.
For $\ell$ large enough, this expression indeed tends to $0$ for
every $r \leq \log(1 / \gamma_\epsilon) \ell$. Take a suitable
$\gamma \leq \log(1/\gamma_\epsilon)$ that is good for every $\ell
\geq 2$.
\item In Section \ref{sec:f_3} take 
\begin{align}
b_j = 
\begin{cases}  
\frac{n_{r-1}}{\ell} & j=r \\
\frac{n_{j-1}}{d} & 1<j<r \\
0 & j=1
\end{cases}
\end{align}
It can be shown that now $q_j^c, q_j^h \leq O(1/d)$ for
any $j < r$. This proves \eqref{eq:proof_all_other_probs} in Theorem
\ref{thm:f_3_probs} and completes the proof of Theorem
\ref{thm:bound_on_f_3}.
\end{itemize}
Note also that in Theorem \ref{thm:4_degenerate} 
a slightly different analysis is needed for large $r$, 
but the stated result remains valid.
\subsection{More on $f_c(\ell, r)$}
Our general upper bound for $f_c(\ell, r)$ is
\begin{equation}
f_c(\ell, r) < \frac{[O(\ell \log{\ell})]^{r+1}}{c^r}
\end{equation}
We have already seen that this bound is tight up to a 
polylogarithmic factor for fixed $c \geq 3$ and $r$.
For other range of the parameters and in particular when $r$ is 
very large 
there is a  result of 
Kierstead, Szemer\'edi and Trotter \cite{kierstead83} providing 
a lower bound for $f_c(\ell, r)$, which is close to being
tight in this range. 
See also \cite{bogdanov13}.  In some cases, however, the
gap between the known upper and lower bounds is large. In
particular, it will be interesting to understand better the
behaviour of $f_c(r,r)$, and of $f_2(\ell,r)$.

\iffalse

showed that 

Here are some other results regarding $f_c(\ell, r)$.

Kierstead, Szemer\'edi and Trotter \cite{kierstead83} showed that 

\[

{f_c}(k(c-1)+1,2kn^{1/k}) \geq n

\]

Under certain conditions on $\ell$ and $r$, this result can be
rewritten as 

\begin{equation} 

\label{eq:lower_kierstead_conc}

f_c(\ell,r) \geq \left({\frac{r(c-1)}{2(\ell-1)}} \right) ^
{\frac{\ell-1}{c-1}}

\end{equation}

Bogdanov \cite{bogdanov13} obtained the upper bound 

\begin{equation} 

{f_c}(k(c-1),r) < \frac{(2rc+1)^k-1}{2r}

\end{equation}

Which is equivalent (under a certain condition on $\ell$) to

\begin{equation} 

\label{eq:upper_bogadnov_conc}

f_c(\ell,r) < \frac{(2rc+1)^{{\frac{\ell}{c-1}}}-1}{2r}

\end{equation}

\eqref{eq:lower_kierstead_conc} and \eqref{eq:upper_bogadnov_conc}
have roughly the same order\footnote{the lower bound actually has a
larger order in $r$, but there is no contradiction since these bounds
are true under different conditions.} in $r$, giving a good
estimation for $f_c(\ell, r)$ when $\ell$ is fixed and $r$ is large.

However, the known upper and lower bounds are far apart when $\ell =
\Theta(r)$. When $c$ is fixed, the lower bound by Kierstead et al.\@
is

\begin{equation}

f_3(\Theta(r), r) \geq \left( \frac{r}{\Theta(r)} \right)^{\Theta(r)}
= \Theta(1)^{\Theta(r)}

\end{equation}

While our upper bound as well as Bogdanov's bound give 

\begin{equation}

f_3(\Theta(r), r) \leq \Theta(r)^{\Theta(r)}

\end{equation}

\fi

The question of obtaining a better 
estimation of $f_c(\ell, r)$ in the general case 
(as well as for fixed $c \geq 3$ and fixed $r$) 
is left as an open problem.

\bibliographystyle{plain}
\bibliography{references}
\end{document}